\documentclass{article}
\usepackage{graphicx} 
\usepackage{amsmath}
\usepackage{amssymb}
\usepackage{amsthm}
\newtheorem{theorem}{Theorem}[section]
\newtheorem{lemma}[theorem]{Lemma}
\newtheorem{proposition}[theorem]{Proposition}
\newtheorem{claim}[theorem]{Claim}
\newtheorem{corollary}[theorem]{Corollary}

\theoremstyle{definition}
\newtheorem{example}[theorem]{Example}

\title{A sharp bound on the integrality gap in the $3$-set cover problem}
\author{Eli Berger\thanks{University of Haifa, \texttt{berger@math.haifa.ac.il}}\,\, and Ron Holzman\thanks{Technion--Israel Institute of Technology, \texttt{holzman@technion.ac.il}}}
\date{}

\begin{document}

\maketitle

\begin{abstract}
Given a hypergraph with edges of size at most $3$, the $3$-set cover problem asks to determine the minimum size of a family of edges which covers the vertex set. As the problem is NP-hard, it is natural to consider its fractional (linear programming) relaxation, which provides a lower bound on the value of the optimal solution. The ratio between the actual value and that of the fractional relaxation is called the integrality gap. A classic bound of Lov\'asz implies that the integrality gap in this problem is at most $11/6$. This has been improved to $5/3$ by Fujito and Okumura. Here we prove that the integrality gap is at most $3/2$, which is best possible. A corollary of this result is that the vertex set of any $3$-uniform, regular hypergraph on $n$ vertices can be covered by $n/2$ (or fewer) edges. This solves the $k=3$ case of a problem of de~A.~Moreira and Kohayakawa. As another application, we derive a certain variant of the Gale-Shapley stable marriage theorem for triples.
\end{abstract}

\section{Introduction}
A \emph{hypergraph} $H$ is a pair $(V,E)$, where $V=V(H)$ is a non-empty finite set of \emph{vertices} and $E=E(H)$ is a family of subsets of $V$ called \emph{edges}. We always assume that the edges are non-empty and every vertex belongs to at least one edge, so that $V=\bigcup E$. We often write edges without braces and commas, e.g., $123$ instead of $\{1,2,3\}$.

The \emph{rank} of $H$ is $\max_{e \in E(H)} |e|$. If $|e|=k$ for every $e \in E(H)$, we say that $H$ is $k$-\emph{uniform}. The \emph{degree} of a vertex $v$ is $|\{e \in E(H): v \in e\}|$. If every $v \in V(H)$ has the same degree $d$, we say that $H$ is $d$-\emph{regular}, or just \emph{regular} without specifying the degree.

An \emph{edge cover} in $H$ is a subfamily $E' \subseteq E(H)$ that covers the entire vertex set, i.e., $V=\bigcup E'$. We will be interested in edge covers that are as small as possible. The \emph{edge covering number} of $H$ is \[ \rho(H) = \min \{|E'|: E' \textrm{ is an edge cover in } H\}.\] The combinatorial optimization problem of determining $\rho(H)$ is usually called the (unweighted) set cover (or covering) problem. When restricted to hypergraphs of rank $k$, it is called the $k$-set cover problem. For $k=2$, i.e., in the case of graphs, computing the edge covering number is easy: it is equivalent to finding the maximum size of a matching. But already for $k=3$, the $3$-set cover problem is known to be NP-hard (e.g., by a reduction from the $3$-dimensional matching problem on Karp's list in~\cite{Ka}). There is a large body of combinatorial opimization literature devoted to approximating the edge covering number; see~\cite{CTF} for a survey.

An efficient lower bound on the edge covering number is obtained by considering its linear programming relaxation. A \emph{fractional edge cover} in $H$ is a function $f: E(H) \to \mathbb{R}_+$ that covers $V(H)$ in the sense that $\sum_{e \ni v} f(e) \ge 1$ for every vertex $v \in V(H)$. Its size $|f|$ is $\sum_{e \in E(H)} f(e)$. The \emph{fractional edge covering number} of $H$ is \[ \rho^*(H) = \min \{|f|: f \textrm{ is a fractional edge cover in } H\}.\] An edge cover is a $\{0,1\}$-valued fractional edge cover, hence $\rho(H) \ge \rho^*(H)$. The ratio $\rho(H)/\rho^*(H)$, measuring the relative distance between the integral solution and its fractional counterpart, is called the edge-cover integrality gap of $H$. An upper bound on the integrality gap over a class of hypergraphs allows to control the quality of the $\rho^*(H)$ approximation to $\rho(H)$ for $H$ in that class. The following classic upper bound due to Lov\'asz~\cite{Lo1} has had many applications in graph and hypergraph theory, combinatorial optimization, additive number theory, and more. (The original formulation of Lov\'asz's bound was in terms of the vertex-cover integrality gap for hypergraphs of a given maximum degree; the form we give here is equivalent by interchanging the roles of vertices and edges while keeping the incidence relation.)
\begin{theorem}[{\cite{Lo1}}] \label{thm:Lo1}
Let $H$ be a hypergraph of rank $k$. Then \[ \frac{\rho(H)}{\rho^*(H)} \le \sum_{i=1}^k \frac{1}{i} \le 1 + \ln k.\]
\end{theorem}
Examples for large $k$ are known that show that Lov\'asz's bound cannot be improved by a multiplicative constant (we cite below, in more precise form, such a result). However, Fujito and Okumura~\cite{FO} have improved the bound by an additive constant: \[ \frac{\rho(H)}{\rho^*(H)} \le \sum_{i=1}^k \frac{1}{i} - \frac{1}{6} \textrm{ \,\, for hypergraphs } H \textrm{ of rank } k \ge 2.\]
This improvement is significant for small values of $k$. For $k=2$ it improves the bound from $\frac{3}{2}$ to $\frac{4}{3}$, which is best possible. Indeed, the triangle graph $K_3$ has $\rho(K_3)=2$ and $\rho^*(K_3)=\frac{3}{2}$. For $k=3$ it improves the bound from $\frac{11}{6}$ to $\frac{5}{3}$. As far as we know, the $\frac{5}{3}$ bound is the best one known prior to this work. Here we focus on the $k=3$ case, and establish the following best possible improvement.
\begin{theorem} \label{thm:3/2}
Let $H$ be a hypergraph of rank $3$. Then \[ \frac{\rho(H)}{\rho^*(H)} \le \frac{3}{2}.\]
\end{theorem}
Before showing the sharpness of this bound, we state two increasingly special cases of Theorem~\ref{thm:3/2}. A fractional edge cover $f: E(H) \to \mathbb{R}_+$ is \emph{exact} if it satisfies the covering constraints as equalities: $\sum_{e \ni v} f(e) = 1$ for all $v \in V(H)$. Suppose that $H$ is a $3$-uniform hypergraph with $|V|=n$, and $f$ is a fractional edge cover in $H$. Then \[ |f| = \frac{1}{3} \sum_{v \in V(H)} \sum_{e \ni v} f(e) \ge \frac{1}{3} \sum_{v \in V(H)} 1 = \frac{n}{3},\] and equality holds if and only if $f$ is exact. Thus, if $H$ has an exact fractional edge cover, then $\rho^*(H) = \frac{n}{3}$. Therefore, Theorem~\ref{thm:3/2} immediately gives the following.
\begin{corollary} \label{cor:Ex}
Let $H$ be a $3$-uniform hypergraph with $|V(H)|=n$ that has an exact fractional edge cover. Then $\rho(H) \le \frac{n}{2}$.
\end{corollary}
Next, suppose that $H$ is a $3$-uniform, $d$-regular hypergraph. Then $f: E(H) \to \mathbb{R}_+$ defined by $f(e) = \frac{1}{d}$ for all $e \in E(H)$ is an exact fractional edge cover in $H$. We get the following special case of Corollary~\ref{cor:Ex}.
\begin{corollary} \label{cor:Reg}
Let $H$ be a $3$-uniform, regular hypergraph with $|V(H)|=n$. Then $\rho(H) \le \frac{n}{2}$.
\end{corollary}
We proceed to give several examples of $3$-uniform, regular hypergraphs that attain equality in Corollary~\ref{cor:Reg}. This will show that Theorem~\ref{thm:3/2} and its two corollaries are sharp.
\begin{example}[$K^{(3)}_4$] \label{ex:1}
The complete $3$-uniform hypergraph on $4$ vertices is also $3$-regular. Two edges are needed to cover all vertices. Thus $\rho(K^{(3)}_4) = 2 = \frac{4}{2}$.
\end{example}
\begin{example}[$TPP(2)$] \label{ex:2}
The truncated projective plane of order $2$ has $6$ vertices and $4$ edges: $135$, $146$, $236$, $245$. It is $2$-regular, and three edges are needed to cover all vertices. Thus $\rho(TPP(2)) = 3 = \frac{6}{2}$. Note that this hypergraph is $3$-partite (with parts $\{1,2\}$, $\{3,4\}$, $\{5,6\}$), showing that the upper bound is sharp even in the $3$-partite case.
\end{example}
Examples satisfying $\rho(H) = \frac{n}{2}$ for arbitrarily large $n$ can be built from any of the above small examples by taking disjoint copies. The following construction gives arbitrarily large connected examples of equality.
\begin{example} \label{ex:3}
Given an even $n \ge 4$, we take a set $V$ of $n$ vertices and partition it into pairs $P_1, P_2, \ldots P_k$, where $k = \frac{n}{2}$ and the pairs are arranged cyclically (so $P_{k+1} = P_1$). We take as edges in $E$ all triples of the form $P_i \cup \{x\}$ where $x \in P_{i+1}$. Then $H=(V,E)$ is $3$-uniform and $3$-regular (in the case $n=4$ we get back Example~\ref{ex:1}). An edge cover of size $k$ can be formed, for example, by taking one triple containing $P_i$ for each $i=1, \ldots ,k$. We claim that $H$ has no smaller edge cover. Indeed, let $E'$ be an edge cover. We associate with $E'$ a cyclic string of length $k$, where the $i$-th entry is $0$, $1$, or $2$ indicating the number of triples in $E'$ that contain $P_i$. The predecessor (in the cyclic sense) of any $0$ entry must be a $2$ in order to satisfy the covering constraints. Hence there must be at least as many $2$ entries as there are $0$ entries, implying that the sum of entries is at least $k$. Thus $|E'| \ge k$ and so $\rho(H) = k = \frac{n}{2}$.
\end{example}
In a related paper, de~A.~Moreira and Kohayakawa~\cite{dK} studied the ratio $\rho(H)/\rho^*(H)$ for hypergraphs that are uniform and regular. They defined the function \[ f(k) = \limsup_{n \to \infty} \max_H \frac{\rho(H)}{\rho^*(H)},\] where the maximum is taken over all $k$-uniform, regular hypergraphs on $n$ vertices. Based on a probabilistic construction for the lower bound and on Lov\'asz's upper bound, they proved that for any large enough fixed $k$ \[ \ln k - 12 \ln\ln k \le f(k) \le 1 + \ln k.\] They showed that $f(2) = \frac{4}{3}$, and posed the open problem of determining $f(k)$ for all $k \ge 3$. Our Corollary~\ref{cor:Reg} together with the above examples of equality solve the case $k=3$ of their problem: $f(3) = \frac{3}{2}$.

The proof of Theorem~\ref{thm:3/2} is quite intricate, and relies on the Gallai-Edmonds decomposition of graphs and some structural properties of factor-critical graphs. Those are recalled or developed in Section~2, and the proof itself appears in Section~3. It is worth mentioning that the proof of Fujito and Okumura's~\cite{FO} previous improvement to Lov\'asz's bound, establishing the $\frac{5}{3}$ upper bound for $k=3$, also used the Gallai-Edmonds decomposition. However, we apply it to a different graph associated with the hypergraph of rank $3$, and use the decomposition in other ways than they did.

In Section~4 we present an application of Theorem~\ref{thm:3/2} to a Gale-Shapley type stable marriage problem, where we consider triples instead of couples.

\section{Preliminaries}
Although Theorem~\ref{thm:3/2} is about hypergraphs of rank $3$, our proof will mostly be in terms of graphs that are associated with such hypergraphs. Here we introduce some terminology and notation for graphs, and recall or develop the main tools that we will use in the proof.

A \emph{graph} $G=(V,E)$ is a $2$-uniform hypergraph. A \emph{matching} in $G$ is a set $M \subseteq E$ of pairwise disjoint edges. The set of vertices covered by the matching $M$ is $\bigcup M \subseteq V$. A \emph{maximum} matching $M$ is one that maximizes $|M|$ among all matchings in $G$. When $|V|$ is even, a \emph{perfect} matching $M$ is one with $\bigcup M = V$. When $|V|$ is odd, a \emph{near-perfect} matching $M$ is one with $|\bigcup M| = |V| - 1$.

Let $G=(V,E)$ be a graph, and let $U \subseteq V$. The \emph{induced subgraph} of $G$ on $U$ is the graph $G[U] = (U, E[U])$ where $E[U]$ consists of all edges in $E$ having both endpoints in $U$. When $U = V \setminus X$, we also write $G[U]$ as $G - X$. The \emph{contraction} of $X$ in $G$ is the graph $G/X$ obtained from $G - X$ by adding a new vertex $v_X$ and an edge $uv_X$ for every $u \in U$ such that $uv \in E$ for at least one $v \in X$.

Let $G$ be a graph, and let $O$ be an odd cycle in $G$. We say that $O$ is \emph{induced} if it has no chords. If $O$ is not induced, an \emph{odd shortening} of $O$ is a cycle $O'$ obtained from $O$ by partitioning its edge set into odd paths $Q_1, Q_2, \ldots , Q_k$, where $k \ge 3$ and each $Q_i$ is either an edge of $O$ or its endpoints form a chord, and replacing each of the latter $Q_i$ by the corresponding chord. Thus, the length of $O'$ is $k$, which is necessarily odd. We will need the following simple lemma.
\begin{lemma} \label{lem:Short}
Let $G$ be a graph, and let $O$ be an odd cycle in $G$. If $O$ is not induced, then it has an odd shortening $O'$ which is induced.
\end{lemma}
\begin{proof} Clearly, $O$ has an odd shortening: if $uv$ is a chord, we can use a partition $Q_1, Q_2, \ldots , Q_k$ consisting of the $u-v$ path on $O$ that is odd and single edges. Let $O'$ be an odd shortening with $k$ as small as possible. Then $O'$ is induced. Indeed, if $xy$ is a chord of $O'$, then replacing the $x-y$ path on $O'$ that is odd by the chord $xy$ yields a shorter odd shortening of $O$.
\end{proof}

\subsection{Factor-critical graphs}
A graph $G$ with an odd number of vertices is \emph{factor-critical} if for every $v \in V(G)$ there exists a near-perfect matching $M$ in $G$ such that $\bigcup M = V(G) \setminus \{v\}$. Factor-critical graphs will play a central role in our proof of Theorem~\ref{thm:3/2}. We will need the following lemma, showing that this property can be preserved under contraction.
\begin{lemma} \label{lem:Ind}
Let $G$ be a factor-critical graph with $|V(G)| > 1$. Then $G$ has an induced odd cycle $O$ such that $G/V(O)$ is factor-critical.
\end{lemma}
\begin{proof} We recall that an odd ear decomposition of a graph $G$ is its presentation as $P_0 + P_1 + \cdots + P_r$, where $P_0$ is a single vertex, and each $P_i$, $i=1, \ldots , r$, is either an odd path that has only its two endpoints in common with $P_0 + \cdots + P_{i-1}$ or an odd cycle that has exactly one vertex in common with $P_0 + \cdots + P_{i-1}$. Lov\'asz~\cite{Lo2} proved that a graph is factor-critical if and only if it has an odd ear decomposition.

Now, let $G$ be factor-critical with $|V(G)| > 1$, and let $P_0 + P_1 + \cdots + P_r$ be an odd ear decomposition of $G$. Since $P_0$ is a single vertex, $P_1$ must be an odd cycle through that vertex. We may assume that $P_1$ is induced. Indeed, if it is not then by Lemma~\ref{lem:Short} it has an odd shortening $P'_1$ which is induced. Say that $Q_1, Q_2, \ldots , Q_{\ell}$ are the corresponding odd paths on $P_1$ of length at least $3$ in the formation of $P'_1$. Then we can get another odd ear decomposition of $G$ as follows. We start with an arbitrary vertex on $P'_1$, then add $P'_1$, followed by the odd paths $Q_1, Q_2, \ldots , Q_{\ell}$. At this point, we have the entire cycle $P_1$ and its chords that appear in $P'_1$. From here on, we add $P_2, P_3, \ldots , P_r$ while skipping those chords of $P_1$ that already appeared in $P'_1$.

So, assume that $P_1$ is induced. Upon contracting $V(P_1)$ into a single vertex $v_{P_1}$, the graph becomes $v_{P_1} + P_2 + \cdots + P_r$, and each $P_i$, $i=2, \ldots , r$, is still of one of the two allowed forms (though odd paths may become odd cycles if the contraction collapses their endpoints into a single vertex). Thus $G/V(P_1)$ has an odd ear decomposition as well, and is therefore factor-critical as required.
\end{proof}

The next lemma allows us to decompose a factor-critical graph into an odd cycle of smaller factor-critical graphs.
\begin{lemma} \label{lem:Odd}
Let $G$ be a graph with $|V(G)| > 1$. Then $G$ is factor-critical if and only if for some positive integer $k$, there exists a partition $U_1, U_2, \ldots , U_{2k+1}$ of $V(G)$ so that:
\begin{itemize}
\item[(a)] $G[U_i]$ is factor-critical for all $i=1,2, \ldots , 2k+1$.
\item[(b)] For $1 \le i \ne j \le 2k+1$, $G$ has at least one edge between $U_i$ and $U_j$ if and only if $i, j$ are cyclically adjacent (i.e., $|i-j| = 1$ or $\{i,j\} = \{1, 2k+1\}$).
\end{itemize}
\end{lemma}
\begin{proof} Assume first that a partition $U_1, U_2, \ldots , U_{2k+1}$ of $V(G)$ satisfying (a) and (b) is given. We have to show that for every $v \in V(G)$ there exists a perfect matching in $G - \{v\}$. Assume w.l.o.g.\ that $v = v_1 \in U_1$. For $i=1,2, \ldots , k$, let $v_{2i}v_{2i+1}$ be an edge of $G$ with endpoints $v_{2i} \in U_{2i}$, $v_{2i+1} \in U_{2i+1}$ (which exists by (b)). For $j=1,2, \ldots , 2k+1$, let $M_j$ be a perfect matching in $G[U_j] - \{v_j\}$ (which exists by (a)). These matchings, together with the edges $v_{2i}v_{2i+1}$, cover all the vertices except $v_1$.

To prove the opposite direction, we use induction on $|V(G)|$. By Lemma~\ref{lem:Ind}, we can find in $G$ an induced odd cycle $O$ such that $G/V(O)$ is factor-critical. If $G/V(O)$ is a single vertex, this means that $G=O$. In this case we can partition $V(G) = V(O)$ into singletons to satisfy the requirements in the lemma.

Thus, we may assume that $G/V(O)$ has more than one vertex. By the induction hypothesis, $G/V(O)$ admits a partition $U_1, U_2, \ldots , U_{2k+1}$ that satisfies the requirements. Let $v_O$ be the new vertex created by the contraction of $V(O)$, and assume w.l.o.g.\ that $v_O \in U_1$. Let $U^*_1 = \bigl( U_1 \setminus \{v_O\}\bigr) \cup V(O)$. We claim that $U^*_1, U_2, \ldots , U_{2k+1}$ is a partition of $V(G)$ that satisfies (a) and (b) with respect to $G$. The only part that requires an argument is the fact that $G[U^*_1]$ is factor-critical.

We need to show that for every $v \in U^*_1$, there is a matching $M$ in $G$ such that $\bigcup M = U^*_1 \setminus \{v\}$. There are two cases to check. If $v \in V(O)$ then we take a perfect matching in $O - \{v\}$ and a perfect matching in $G[U_1] - \{v_O\}$ (which exists because $\bigl( G/V(O)\bigr)[U_1]$ is factor-critical). Together they give the required matching $M$. Now suppose that $v \in U_1 \setminus \{v_O\}$. Then there is a perfect matching $M_1$ in $\bigl( G/V(O)\bigr)[U_1] - \{v\}$, and one of its edges is of the form $uv_O$ where $u \in U_1 \setminus \{v, v_O\}$. By the definition of contraction, $G$ has an edge $uw$ for some $w \in V(O)$. Replacing $uv_O$ in $M_1$ by $uw$, and adding a perfect matching in $O - \{w\}$, we get the required matching $M$.
\end{proof}

\subsection{The Gallai-Edmonds decomposition}
Gallai~\cite{Ga63, Ga64} and Edmonds~\cite{Ed} introduced a canonical decomposition that applies to any graph and can be used to describe the family of maximum matchings in it. We recall here the definition of this decomposition and some of its properties (see~\cite[Theorem 3.2.1]{LP}).

Let $G$ be a graph. We call a vertex $v \in V(G)$ \emph{dispensable} in $G$ if there is a maximum matching $M$ in $G$ with $v \notin \bigcup M$. We decompose $V(G)$ into three (possibly empty) parts as follows:
\begin{eqnarray*}
D & = & \{ v \in V(G): v \textrm{ is dispensable in } G\} \\
A & = & \{ u \in V(G) \setminus D: u \textrm{ is adjacent to at least one vertex in } D\} \\
C & = & V(G) \setminus (D \cup A)
\end{eqnarray*}
If $D \ne \emptyset$, we further decompose $D$ into the vertex sets $D_1, D_2, \ldots , D_r$ of the connected components of $G[D]$. If $D = \emptyset$ set $r=0$.
\begin{theorem}[{\cite{Ga63, Ga64, Ed}}] \label{thm:GE}
Let $G$ be a graph, and let $D = D_1 \cup D_2 \cup \cdots \cup D_r$, $A$, $C$ be the vertex sets defined above. Then:
\begin{itemize}
    \item[(a)] $G[D_i]$ is factor-critical for all $i=1,2, \ldots , r$.
    \item[(b)] There exists a matching in $G$ of size $|A|$, that matches the vertices in $A$ to vertices in pairwise distinct $D_j$'s. Moreover, given any $i \in \{1,2, \ldots , r\}$, such a matching exists that avoids $D_i$.
    \item[(c)] $G[C]$ has a perfect matching.
    \item[(d)] The family of all maximum matchings $M$ in $G$ can be constructed as follows. Choose any matching $M_A$ as described in part (b). For every $D_j$ that has a vertex covered by $M_A$, choose a near-perfect matching in $G[D_j]$ that avoids that vertex. For every other $D_i$ choose an arbitrary near-perfect matching in $G[D_i]$. Add to the above any perfect matching in $G[C]$. This gives a maximum matching $M$ in $G$ that covers all but $r - |A|$ vertices, one in each $D_i$ that is not touched by $M_A$.
\end{itemize}
\end{theorem}

We now list, for later use, a few immediate consequences of Theorem~\ref{thm:GE}. In each of them, we refer to the Gallai-Edmonds decomposition of $G$ as defined above.
\begin{corollary} \label{cor:NBR}
If $v \in D_i$ then any neighbor of $v$ in $G$ is either in $D_i$ or in $A$.
\end{corollary}
\begin{corollary} \label{cor:BP}
Suppose $D \ne \emptyset$. Consider the bipartite graph $\Gamma_G$ with $r$ vertices on one side (corresponding to $D_1, D_2, \ldots , D_r$) and the vertices of $A$ on the other side, where the $D_j$-vertex is adjacent to $u \in A$ if and only if at least one vertex in $D_j$ is adjacent to $u$ in $G$. Then $\Gamma_G$ has a matching that covers the $A$ side, and this matching can be chosen so that it avoids any given $D_i$-vertex. In particular, $|A| < r$.
\end{corollary}
\begin{corollary} \label{cor:MM}
Suppose $v \in D_i$ and $M$ is a maximum matching in $G$ with $v \notin \bigcup M$. Then $M$ contains a perfect matching in $G[D_i] - \{v\}$.
\end{corollary}

In addition to what we proved in the previous subsection, we will need another fact about the behavior of factor-critical graphs under contraction. Its proof uses the Gallai-Edmonds decomposition.
\begin{lemma} \label{lem:PM}
Let $G$ be a factor-critical graph, and let $X \subseteq V(G)$ be such that $G - X$ has a perfect matching. Then $G/X$ is factor-critical.
\end{lemma}
\begin{proof} Let $D = D_1 \cup D_2 \cup \cdots \cup D_r$, $A$, $C$ be the Gallai-Edmonds decomposition of $G/X$. Because $G - X$ has a perfect matching, so does $(G/X) - \{v_X\}$. It follows that $v_X$ is dispensable in $G/X$, and therefore $v_X \in D$. Assume w.l.o.g.\ that $v_X \in D_1$. Now let $D^*_1 = \bigl( D_1 \setminus \{v_X\}\bigr) \cup X$, and consider the partition $D^*_1, D_2, \ldots , D_r, A, C$ of $V(G)$. By Corollary~\ref{cor:NBR}, vertices in $D_i$, $i=1,2, \ldots , r$, have $G/X$-neighbors only in $D_i \cup A$. By the definition of contraction, the same remains true for $D^*_1, D_2, \ldots , D_r$ with respect to $G$-neighbors.

Assume first that $A \ne \emptyset$, and let $u \in A$. Since $G$ is factor-critical, there is a perfect matching in $G - \{u\}$. Because $D^*_1, D_2, \ldots , D_r$ all have odd cardinality, and have outer neighbors only in $A$, each of them must have a vertex matched to some vertex in $A \setminus \{u\}$. But by Corollary~\ref{cor:BP} we have $|A| < r$, so this is impossible.

Thus, we conclude that $A = \emptyset$. Now the sets $D^*_1, D_2, \ldots , D_r, C$ are disconnected from each other in $G$, yet $G$ must be a connected graph (otherwise it would not be factor-critical). It follows that $r=1$ and $C = \emptyset$. Returning to the Gallai-Edmonds decomposition of $G/X$, it reduces to just $D_1$, implying that $G/X$ itself is factor-critical.
\end{proof}

\subsection{Matchings in bipartite graphs}
We recall here some facts about matchings in bipartite graphs. Let $G$ be a bipartite graph with bipartition $(X, Y)$. For $S \subseteq X$ we denote by $N(S) \subseteq Y$ the set of vertices having at least one neighbor in $S$. We say that a set $U \subseteq V(G)$ is \emph{matchable} if there exists a matching $M$ in $G$ with $U \subseteq \bigcup M$. Hall's theorem says that $X$ is matchable if and only if $|N(S)| \ge |S|$ for every $S \subseteq X$. We will need the following well-known generalization of this theorem, attributed to Ore~\cite{Ore}. Define the \emph{deficiency} of $G$ with respect to $X$ to be \[ \mathrm{def}_X(G) = \max \{|S| - |N(S)|: S \subseteq X\}.\]
\begin{theorem}[{\cite{Ore}}] \label{thm:DEF}
Let $G$ be a bipartite graph with bipartition $(X, Y)$. Then $X$ has a matchable subset $X'$ of size $|X'| = |X| - \mathrm{def}_X(G)$.
\end{theorem}
We will also use a result of Mendelsohn and Dulmage~\cite{MD} allowing to combine matchability of subsets of different sides.
\begin{theorem}[{\cite{MD}}] \label{thm:MD}
Let $G$ be a bipartite graph with bipartition $(X, Y)$. If each of the sets $X' \subseteq X$ and $Y' \subseteq Y$ is matchable, then $X' \cup Y'$ is matchable.
\end{theorem}

\section{Proof of Theorem~\ref{thm:3/2}}
Let $H = (V, E)$ be a hypergraph of rank $3$. We use the notation \[ E_i = \{ e \in E: |e| = i\}, \,\, i = 1,2,3\]
to refer to the edges of each possible size. We also define the \emph{shadow} of $E_3$: \[ \partial E_3 = \{ e \setminus \{v\}: v \in e \in E_3\}.\] This is the family of subsets of size $2$ of edges of size $3$. We will consider the following graph associated with $H$: \[G_H = (V, E_2 \cup \partial E_3).\] For a function $f: E \to \mathbb{R}_+$ we write \[ |f|_- = \sum_{e \in E_1 \cup E_2} f(e).\] This is the total weight assigned by $f$ to edges of size less than $3$. We will deduce Theorem~\ref{thm:3/2} from the following proposition.
\begin{proposition} \label{pro:Cover}
Let $H = (V, E)$ be a hypergraph of rank $3$. Assume the following two conditions hold:
\begin{itemize}
    \item[(a)] The graph $G_H$ is factor-critical.
    \item[(b)] $H$ has an exact fractional edge cover $f: E \to \mathbb{R}_+$ such that $|f|_- < 1$.
\end{itemize}
Then $\rho(H) \le \frac{|V|-1}{2}$.
\end{proposition}
Note that condition (a) implies that $|V|$ is odd. By taking a near-perfect matching in $G_H$ and one additional edge covering the missing vertex, we can cover $V$ by $\frac{|V|-1}{2} + 1$ edges of $G_H$. Since every edge of $G_H$ is either an edge of $H$ or is contained in one, it immediately follows that $\rho(H) \le \frac{|V|-1}{2} + 1$. Our task in proving Proposition~\ref{pro:Cover} is `just' to save $1$ in this upper bound, by using edges of size $3$ and the full power of the assumptions.

In Subsection~3.1 we will show how Theorem~\ref{thm:3/2} follows from Proposition~\ref{pro:Cover}. Then, in the next two subsections, we will prove Proposition~\ref{pro:Cover}.

\subsection{Proof of Theorem~\ref{thm:3/2} from Proposition~\ref{pro:Cover}}
Let $f: E \to \mathbb{R}_+$ be a fractional edge cover. For $v \in V$ we write \[ |f|_{v^+} = \sum_{e \ni v} f(e) - 1.\] Bearing in mind that $\sum_{e \ni v} f(e) \ge 1$ for every $v$, we think of $|f|_{v^+}$ as the extra coverage of $f$ at the vertex $v$. We further write \[ |f|_+ = \sum_{v \in V} |f|_{v^+}.\] This is the total extra coverage of $f$. Note that $f$ is exact if and only if $|f|_+ = 0$.

We first observe that if Proposition~\ref{pro:Cover} is true, then it remains true when condition (b) is relaxed to: \[ (b')\,\, H\,\, has\,\, a\,\, fractional\,\, edge\,\, cover\,\,  f: E \to \mathbb{R}_+\,\,  such\,\, that\,\,  |f|_- + |f|_+ < 1.\] Indeed, assume that $f$ satisfies $|f|_- + |f|_+ < 1$. For a vertex $v$ with $|f|_{v^+} > 0$, we reassign a total weight of $|f|_{v^+}$ from edges $e \in E_3$ containing $v$ to the corresponding $e \setminus \{v\}$. This reduces $|f|_+$ by $|f|_{v^+}$ while increasing $|f|_-$ by the same amount, thus preserving $|f|_- + |f|_+$. Repeating this for every $v$ with $|f|_{v^+} > 0$, we end up with a modified function $\tilde{f}$ such that $|\tilde{f}|_- = |f|_- + |f|_+ < 1$ and $|\tilde{f}|_+ = 0$. In the course of doing this, we may have created edges of size $2$ that were not in $H$, so let $\tilde{H}$ be the hypergraph obtained by adding those edges to $H$. Now $\tilde{f}$ is an exact fractional edge cover in $\tilde{H}$ satisfying condition (b). Note that $G_{\tilde{H}} = G_H$, so condition (a) still holds. Proposition~\ref{pro:Cover} implies that $\rho(\tilde{H}) \le \frac{|V|-1}{2}$. Because every edge of $\tilde{H}$ is contained in an edge of $H$, we conclude that $\rho(H) \le \frac{|V|-1}{2}$ as desired.

Assuming Proposition~\ref{pro:Cover}, we proceed to prove Theorem~\ref{thm:3/2}. Let $H = (V, E)$ be a hypergraph of rank $3$, and let $f: E \to \mathbb{R}_+$ be a fractional edge cover in $H$. We have to prove that $\rho(H) \le \frac{3}{2} |f|$.

We first prove this in the special case when $H$ is $3$-uniform. Consider the graph $G_H$ associated with $H$, which is $(V, \partial E)$ in this case. Let $D = D_1 \cup D_2 \cup \cdots \cup D_r$, $A$, $C$ be its Gallai-Edmonds decomposition. For each $i=1,2, \ldots , r$ we consider a graph and a hypergraph defined on the vertex set $D_i$. The graph is $G_H[D_i]$, it is factor-critical by Theorem~\ref{thm:GE}(a). The hypergraph is $H_{D_i}$, the restriction of $H$ to $D_i$, having the edge set
\begin{equation} \label{eq:rest}
E(H_{D_i}) = \{ e \cap D_i: e \in E, e\cap D_i \ne \emptyset\}.
\end{equation}
Observe that $G_{H_{D_i}} = G_H[D_i]$. We also define $f_i: E(H_{D_i}) \to \mathbb{R}_+$ by
\begin{equation} \label{eq:frest}
f_i(e') = \sum \{f(e): e \in E, e \cap D_i = e'\}.
\end{equation}
Note that $f_i$ is a fractional edge cover in $H_{D_i}$, and $|f_i|_{v^+} = |f|_{v^+}$ for every $v \in D_i$. Observe also that while $|f|_- = 0$ by the assumption that $H$ is $3$-uniform, $|f_i|_-$ may be positive: it is given by
\begin{equation} \label{eq:minus}
|f_i|_- = \sum \{ f(e): e \in E, e \cap D_i \ne \emptyset ,e \setminus D_i \ne \emptyset \}.
\end{equation}
We distinguish between those $D_i$ for which $|f_i|_- + |f_i|_+$ is high (at least $1$) and those for which it is low (less than $1$). We assume w.l.o.g.\ that
\begin{eqnarray*}
|f_i|_- + |f_i|_+ \ge 1\,\, & \textrm{for} & 1 \le i \le h, \\
|f_i|_- + |f_i|_+ < 1\,\, & \textrm{for} & h < i \le r,
\end{eqnarray*}
where $0 \le h \le r$. For every $h < i \le r$ we can apply Proposition~\ref{pro:Cover} in its generalized form (with conditions (a) and (b')) to conclude that $\rho(H_{D_i}) \le \frac{|D_i|-1}{2}$.

In order to handle $D_1, D_2, \ldots , D_h$, we consider the corresponding vertices $1, 2, \ldots , h$ in the bipartite graph $\Gamma_{G_H}$ defined in Corollary~\ref{cor:BP}. We claim that for every $S \subseteq \{1,2, \ldots , h\}$ we have
\begin{equation} \label{eq:def}
|S| - |N(S)| \le |f|_+,
\end{equation}
where $N(S) \subseteq A$ is the set of neighbors of vertices of $S$ in $\Gamma_{G_H}$. Indeed, given $i \in S$, for every $e \in E$ such that $e \cap D_i \ne \emptyset$, we have $e \setminus D_i \subseteq N(S)$ by Corollary~\ref{cor:NBR} and the definition of $\Gamma_{G_H}$. It follows from (\ref{eq:minus}) that each $i \in S$ contributes at least $|f_i|_-$ to the sum $\sum_{u \in N(S)} \sum_{e \ni u} f(e)$. Adding up these contributions we get \[ \sum_{i \in S} |f_i|_- \le \sum_{u \in N(S)} \sum_{e \ni u} f(e) = \sum_{u \in N(S)} (1 + |f|_{u^+}) = |N(S)| + \sum_{u \in N(S)} |f|_{u^+}.\] Furthermore, \[ \sum_{i \in S} |f_i|_+ = \sum_{i \in S} \sum_{v \in D_i} |f_i|_{v^+} = \sum_{v \in \bigcup _{i \in S} D_i} |f|_{v^+}.\] We conclude that \[ |S| \le \sum_{i \in S} (|f_i|_- + |f_i|_+) \le |N(S)| + \sum_{u \in N(S)} |f|_{u^+} +\sum_{v \in \bigcup_{i \in S} D_i} |f|_{v^+} \le |N(S)| + |f|_+,\] proving inequality~(\ref{eq:def}).

Since (\ref{eq:def}) holds for every $S \subseteq \{1,2, \ldots , h\}$, it follows from Theorem~\ref{thm:DEF} that $\{1,2, \ldots , h\}$ has a subset $B$ of size at least $h - |f|_+$ that is matchable in $\Gamma_{G_H}$. We also know from Corollary~\ref{cor:BP} that $A$ is matchable in $\Gamma_{G_H}$. Using Theorem~\ref{thm:MD} we obtain a matching $M$ in $\Gamma_{G_H}$ such that $A \cup B \subseteq \bigcup M$. Let $I \subseteq \{1,2, \ldots , r\}$ be the set of vertices on this side that are covered by $M$, and let $I^c = \{1,2, \ldots , r\} \setminus I$. As $B \subseteq I$, we have $|I^c \cap \{1,2, \ldots , h\}| \le |f|_+$.

Returning from $\Gamma_{G_H}$ to $G_H$ itself, we can replace each edge $iu$ in $M$ by an edge $v_iu$ of $G_H$, where $v_i$ is an arbitrarily chosen neighbor of $u$ in $D_i$. Let $M'$ be the obtained matching in $G_H$. By adding to $M'$ a perfect matching of $G_H[D_i] - \{v_i\}$ for each $i \in I$, and a perfect matching of $G_H[C]$, we get a matching $M''$ in $G_H$ that covers $V \setminus \bigcup_{i \in I^c} D_i$.

Now, we can find an edge cover in $H$ as follows. We replace each edge in $M''$ by an edge of $H$ containing it. This uses $\frac{1}{2} |V \setminus \bigcup_{i \in I^c} D_i|$ edges. For each $i \in I^c \cap \{1,2, \ldots , h\}$ there are $\frac{|D_i|-1}{2} + 1$ edges of $G_H$ that cover $D_i$, and we replace them by edges of $H$ containing them. This uses at most $\sum_{i \in I^c \cap \{1,2, \ldots , h\}} \frac{|D_i|+1}{2}$ edges. For each $i \in I^c \setminus \{1,2, \ldots , h\}$ there are $\frac{|D_i|-1}{2}$ edges of $H_{D_i}$ that cover $D_i$ (as shown above), and we replace them by edges of $H$ that contain them. This uses $\sum_{i \in I^c \setminus \{1,2, \ldots , h\}} \frac{|D_i|-1}{2}$ edges. Summing up, this gives
\begin{eqnarray*}
\rho(H) & \le & \frac{1}{2} |V \setminus \bigcup _{i \in I^c} D_i| + \sum_{i \in I^c \cap \{1,2, \ldots , h\}} \frac{|D_i|+1}{2} + \sum _{i \in I^c \setminus \{1,2, \ldots , h\}} \frac{|D_i|-1}{2} \\
& = & \frac{1}{2} \bigl( |V| + |I^c \cap \{1,2, \ldots , h\}| - |I^c \setminus \{1,2, \ldots , h\}|\bigr) \\
& \le & \frac{1}{2}(|V| + |f|_+) \\
& = & \frac{1}{2} \sum_{v \in V} (1 + |f|_{v^+}) \\
& = & \frac{1}{2} \sum_{v \in V} \sum_{e \ni v} f(e) \\
& = & \frac{1}{2} \sum_{e \in E} 3f(e) \\
& = & \frac{3}{2} |f|.
\end{eqnarray*}
This completes the proof of Theorem~\ref{thm:3/2} from Proposition~\ref{pro:Cover} in the case of $3$-uniform hypergraphs.

We treat the general rank $3$ case by reducing it to the $3$-uniform case. Let $H = (V, E)$ and a fractional edge cover $f: E \to \mathbb{R}_+$ be given. If there are two edges $e \subsetneq e'$ in $E$, we reassign any weight on $e$ to $e'$. Doing this for all such pairs, and removing from $E$ any edge with weight $0$, we may assume that there is no containment among edges in $E$, and all of them have positive weight. If there is a singleton edge $\{v\}$, then it is the only one containing $v$, so its weight must be $1$ (or more). We remove the vertex $v$ and the edge $\{v\}$, find a good edge cover in the remaining hypergraph, and add $\{v\}$ to get a good cover in $H$. Thus, we may assume that there is no singleton edge and there is at least one edge of size $2$ (otherwise $H$ is $3$-uniform).

Let $m > 0$ be the total weight of the edges of size $2$, and let $k$ be an integer so that $km \ge 1$. We form a new $3$-uniform hypergraph $kH$ on the vertex set $\bigcup _{i=1}^k V_i \cup \{w\}$, where each $V_i$ is a disjoint copy of $V$, and $w$ is a new vertex. In each $V_i$ we have copies of all edges of $H$ of size $3$, with their original weight. For each edge $uv$ of size $2$ in $H$, we place $k$ edges $u_iv_iw$ where $u_iv_i$ is its copy in $V_i$, and assign to each of them the weight of $uv$. Let $kf$ be the resulting weight function on $E(kH)$. Clearly, $kf$ is a fractional edge cover in $kH$, and $|kf| = k|f|$. Also $\rho(kH) \ge k\rho(H)$. Applying Theorem~\ref{thm:3/2} to $kH$ gives $\rho(kH) \le \frac{3}{2} |kf|$. We deduce that $k\rho(H) \le \frac{3}{2} k|f|$ and therefore $\rho(H) \le \frac{3}{2} |f|$ as required.

\subsection{The structure of a minimal counterexample to Proposition~\ref{pro:Cover}}
In view of the previous subsection, it remains to prove Proposition~\ref{pro:Cover}. Assuming that it is false, we consider a counterexample $H = (V, E)$ with $|V|$ as small as possible. In this subsection we prove Lemma~\ref{lem:Cons}. Roughly speaking, it identifies a subgraph of $G_H$ which is already factor-critical, and can be constructed in a particular way. Then, in the next subsection, we will obtain a contradiction by showing that such a structured counterexample cannot exist.
\begin{lemma} \label{lem:Cons}
Assume that $H = (V, E)$ is a counterexample to Proposition~\ref{pro:Cover}, and $|V|$ is minimal among such counterexamples. Then there exists a vertex $x \in V$ and there is some $t \ge 1$ and a sequence $G^{(1)}, G^{(2)}, \ldots , G^{(t)}$ of subgraphs $G^{(k)} = (V, E^{(k)})$ of $G_H$ satisfying:
\begin{itemize}
    \item[(a)] $E^{(1)} = \{e \in E_2: x \in e\} \cup \partial E_3$.
    \item[(b)] $G^{(t)}$ is factor-critical.
    \item[(c)] For each $1 \le k < t$: \begin{itemize}
        \item[(c1)] The Gallai-Edmonds decomposition of $G^{(k)}$ is of the form $D^{(k)} = D_1^{(k)} \cup D_2^{(k)} \cup \cdots \cup D_{r^{(k)}}^{(k)}$, $A^{(k)}$, $C^{(k)}$ where $|A^{(k)}| = r^{(k)} - 1 \ge 1$.
        \item[(c2)] $E^{(k+1)}$ is obtained from $E^{(k)}$ by adding one edge $e_k \in E_2 \setminus E^{(k)}$.
        \item[(c3)] $e_k$ has one endpoint in $D_{j^{(k)}}^{(k)}$ and the other outside it, where $1 \le j^{(k)} \le r^{(k)}$ is such that $|D_{j^{(k)}}^{(k)}| \ge 3$ and $x \notin D_{j^{(k)}}^{(k)}$.
        \end{itemize}
    \end{itemize}
\end{lemma}
\begin{proof} Because $H$ is a counterexample to Proposition~\ref{pro:Cover} it satisfies its two conditions: $G_H$ is factor-critical, and we are given an exact fractional edge cover $f: E \to \mathbb{R}_+$ such that $|f|_- < 1$. Yet $\rho(H) > \frac{|V|-1}{2}$.

For a vertex $v \in V$ we write \[ E_{12}(v) = \{ e \in E_1 \cup E_2: v \in e\}\,\, \textrm { and } \,\,|f|_{v^-} = \sum_{e \in E_{12}(v)} f(e).\] We choose $x \in V$ to be a vertex that maximizes $|f|_{v^-}$ among all $v \in V$. If there are several maximizers, we choose one of them arbitrarily.

We define $E^{(1)}$ as in part (a) of the lemma. If $G^{(1)} = (V, E^{(1)})$ is factor-critical, we set $t=1$ and are done (note that part (c) is vacuous for $t=1$). Otherwise, we iteratively add edges from $E_2$ as indicated in part (c), until we obtain a graph $G^{(t)}$ that is factor-critical, and then we stop. The process must terminate with a factor-critical graph, because $G_H$ is factor-critical by assumption. What we need to show is that the iterations can be carried out. Namely, given $G^{(k)} = (V, E^{(k)})$ that is not factor-critical, with $E^{(1)} \subseteq E^{(k)} \subsetneq E(G_H)$, we have to show that the Gallai-Edmonds decomposition of $G^{(k)}$ has the form indicated in (c1), and there exists an edge $e_k \in E_2 \setminus E^{(k)}$ satisfying (c3).

From now on, we fix such $k$ and omit the superscript $(k)$. Thus, we consider a non-factor-critical graph $G$ with $E^{(1)} \subseteq E(G) \subsetneq E(G_H)$, and its Gallai-Edmonds decomposition $D = D_1 \cup D_2 \cup \cdots \cup D_r$, $A$, $C$. For $i=1,2, \ldots , r$ we define the restriction $H_{D_i}$ of $H$ to $D_i$ as in (\ref{eq:rest}), and $f_i: E(H_{D_i}) \to \mathbb{R}_+$ as in (\ref{eq:frest}). Since $f$ is an exact fractional edge cover in $H$, so is $f_i$ in $H_{D_i}$. We distinguish between those $D_i$ for which $|f_i|_-$ is high (at least $1$) and those for which it is low (less than $1$). We assume w.l.o.g.\ that
\begin{eqnarray*}
|f_i|_-  \ge 1\,\, & \textrm{for} & 1 \le i \le h, \\
|f_i|_-  < 1\,\, & \textrm{for} & h < i \le r,
\end{eqnarray*}
where $0 \le h \le r$.
Our proof proceeds via a series of claims.
\begin{claim} \label{cl:Low}
For every $h < i \le r$ there exists a family of at most $\frac{|D_i|-1}{2}$ edges of $H$ whose union covers $D_i$.
\end{claim}
\begin{proof} Fix $h < i \le r$, and consider $H_{D_i}$. As $G[D_i]$ is factor-critical, so is its supergraph $G_{H_{D_i}}$. Thus $H_{D_i}$ satisfies condition (a) of Proposition~\ref{pro:Cover}, and our assumption that $|f_i|_- < 1$ implies that condition (b) holds, too. Observe also that $|D_i| < |V|$, otherwise $G$ itself would be factor-critical, contradicting our assumption. By the minimality of $H$ as a counterexample to Proposition~\ref{pro:Cover}, we conclude that $\rho(H_{D_i}) \le \frac{|D_i|-1}{2}$. As every edge of $H_{D_i}$ is contained in an edge of $H$, this implies the claim.
\end{proof}
\begin{claim} \label{cl:Edge}
For every $1 \le i \le h$ there exists an edge of $G$ with one endpoint in $D_i$ and the other in $A$.
\end{claim}
\begin{proof} Fix $1 \le i \le h$. Since $|f|_- < 1$ but $|f_i|_- \ge 1$, there must exist an edge $e \in E_3$ such that $1 \le |e \cap D_i| < 3$. Such an edge of size $3$ contributes to $\partial E_3$ two edges of size $2$, each having an endpoint in $D_i$ and another outside it. As $\partial E_3 \subseteq E^{(1)} \subseteq E(G)$, these are also edges of $G$. By Corollary~\ref{cor:NBR}, the endpoint outside $D_i$ must be in $A$, proving the claim.
\end{proof}
We recall the bipartite graph $\Gamma_G$ defined in Corollary~\ref{cor:BP}. In terms of $\Gamma_G$, Claim~\ref{cl:Edge} means that every $1 \le i \le h$ has at least one neighbor in $A$. The next few claims give more information about $\Gamma_G$.
\begin{claim} \label{cl:No}
The set $\{1,2, \ldots , h\}$ is not matchable in $\Gamma_G$.
\end{claim}
\begin{proof} Suppose that $\{1,2, \ldots , h\}$ is matchable. By Corollary~\ref{cor:BP} $A$ is matchable. Hence Theorem~\ref{thm:MD} provides a matching $M$ in $\Gamma_G$ such that $\{1,2, \ldots , h\} \cup A \subseteq \bigcup M$. Let $I \subseteq \{1,2, \ldots , r\}$ be the set of vertices on this side that are covered by $M$, and let $I^c = \{1,2, \ldots , r\} \setminus I$. Then $I^c \subseteq \{h+1, \ldots , r\}$. Also, since $|A| < r$ by Corollary~\ref{cor:BP}, $I^c \ne \emptyset$.

Returning from $\Gamma_G$ to $G$, we can replace each edge $iu$ in $M$ by an edge $v_iu$ of $G$, where $v_i \in D_i$. Let $M'$ be the obtained matching in $G$. By adding to $M'$ a perfect matching of $G[D_i] - \{v_i\}$ for each $i \in I$, and a perfect matching of $G[C]$, we get a matching $M''$ in $G$ that covers $V \setminus \bigcup_{i \in I^c} D_i$. Now we replace every edge of $M''$ by an edge of $H$ containing it. For each $i \in I^c$ we find, using Claim~\ref{cl:Low}, a family of at most $\frac{|D_i|-1}{2}$ edges of $H$ whose union covers $D_i$. Altogether, since $I^c \ne \emptyset$, this yields a family of at most $\frac{|V|-1}{2}$ edges of $H$ covering $V$. This contradicts the assumption that $H$ is a counterexample to Proposition~\ref{pro:Cover}.
\end{proof}
In particular, Claim~\ref{cl:No} implies that $\{1,2, \ldots , h\} \ne \emptyset$. Using also Claim~\ref{cl:Edge}, we deduce that $h \ge 2$ and $|A| \ge 1$.
\begin{claim} \label{cl:SD}
Let $S \subseteq \{1,2, \ldots , h\}$. Then in $\Gamma_G$ we have \[ |N(S)| > |S| - 2 + |f|_{x^-}.\] In particular, $|N(S)| \ge |S| - 1$. Moreover, if for some $i \in S$ we have $|D_i| = 1$ then $|N(S)| \ge |S|$.
\end{claim}
\begin{proof} As $f$ is an exact fractional edge cover, we can write \[ |N(S)| = \sum_{u \in N(S)} \sum_{e \ni u} f(e) = \sum_{e \in E} |e \cap N(S)| f(e).\] For $i \in S$, we denote by $c_i$ the contribution of edges $e$ that intersect $D_i$ to the above sum. Note that no $e$ with $e \cap N(S) \ne \emptyset$ can intersect two distinct $D_i$'s, because then it would have to be of size $3$ and contain an edge of $G$ between distinct $D_i$'s. Thus
\begin{equation} \label{eq:nsci}
|N(S)| \ge \sum_{i \in S} c_i.
\end{equation}

In order to estimate $c_i$ from below, we write
\begin{eqnarray} \label{eq:twosums}
|f_i|_- & = & \sum \{ f(e): e \in E_3, 1 \le |e \cap D_i| \le 2\} \nonumber \\  & + & \sum \{ f(e): e \in E_1 \cup E_2, e \cap D_i \ne \emptyset\}.
\end{eqnarray}
Every edge $e \in E_3$ with $1 \le |e \cap D_i| \le 2$ contains at least one vertex in $N(S)$, hence contributes at least $f(e)$ to $c_i$. As $|f_i|_- \ge 1$, we obtain \[ c_i \ge 1 - \sum \{ f(e): e \in E_1 \cup E_2, e \cap D_i \ne \emptyset \}.\] Adding up these inequalities we get
\begin{eqnarray} \label{eq:cis}
\sum_{i \in S} c_i & \ge & \sum_{i \in S} \bigl( 1 - \sum \{ f(e): e \in E_1 \cup E_2, e \cap D_i \ne \emptyset \} \bigr) \nonumber \\
& = & |S| - \sum_{e \in E_1 \cup E_2} |\{ i \in S: e \cap D_i \ne \emptyset \}| f(e).
\end{eqnarray}
Clearly, an edge $e \in E_1 \cup E_2$ can intersect at most two $D_i$'s. Moreover, if $x \in e$ then $e$ can intersect at most one $D_i$: this is obvious when $e \in E_1$, and if $e \in E_2$ then by definition $e \in E^{(1)} \subseteq E(G)$ so it cannot connect two $D_i$'s. Therefore
\begin{eqnarray} \label{eq:e12}
\sum_{e \in E_1 \cup E_2} |\{ i \in S: e \cap D_i \ne \emptyset \}| f(e) & \le & \sum_{e \in E_1 \cup E_2} 2f(e) - \sum_{e \in E_{12}(x)} f(e) \nonumber \\
& = & 2|f|_- - |f|_{x^-} \nonumber \\
& < & 2 - |f|_{x^-}.
\end{eqnarray}
Putting together (\ref{eq:nsci}), (\ref{eq:cis}) and (\ref{eq:e12}) we obtain \[ |N(S)| > |S| - 2 + |f|_{x^-},\] proving the first part of the claim.

Suppose now that there is an $i \in S$ with $|D_i| = 1$, say $D_i = \{v\}$. For this $i$, each edge $e$ that appears in the first sum in (\ref{eq:twosums}) contains $v$ and two vertices in $N(S)$, hence contributes $2f(e)$ to $c_i$. This allows us to double the lower bound on this $c_i$: \[ c_i \ge 2 \bigl( 1 - \sum \{ f(e): e \in E_1 \cup E_2, e \cap D_i \ne \emptyset \}\bigr) = 2 (1 - |f|_{v^-}).\] Taking this into account, we can add $1 - |f|_{v^-}$ to the calculation above and get \[ |N(S)| > |S| - 2 + |f|_{x^-} + 1 - |f|_{v^-} \ge |S| - 1,\] where the last inequality follows from the choice of $x$ as maximizing $|f|_{v^-}$. It follows that $|N(S)| \ge |S|$ in this case, as claimed.
\end{proof}
\begin{claim} \label{cl:One}
There exists a subset $S \subseteq \{1,2, \ldots , h\}$ of size $|S| \ge 2$ such that $|N(S)| = |S| - 1$ and $|D_i| \ge 3$ for every $i \in S$.
\end{claim}
\begin{proof} By Claim~\ref{cl:SD} all subsets $S \subseteq \{1,2, \ldots , h\}$ satisfy $|N(S)| \ge |S| - 1$. If none of these inequalities holds as an equality, then by Hall's theorem $\{1,2, \ldots , h\}$ is matchable in $\Gamma_G$, contradicting Claim~\ref{cl:No}. Hence we can find an $S \subseteq \{1,2, \ldots , h\}$ such that $|N(S)| = |S| - 1$. Claim~\ref{cl:Edge} rules out the possibility $|S| = 1$, hence $|S| \ge 2$. The second part of Claim~\ref{cl:SD} precludes any $i \in S$ with $|D_i| = 1$.
\end{proof}
\begin{claim} \label{cl:Ar}
$|A| = r - 1.$
\end{claim}
\begin{proof} By Corollary~\ref{cor:BP} we know that $|A| < r$. Moreover, since $|V|$ and all $|D_i|$ are odd while $|C|$ is even, $|A|$ and $r$ must have different parities. Thus, if the claim is wrong then $r - |A| \ge 3$. Using Theorem~\ref{thm:DEF} and Claim~\ref{cl:SD}, we can find a subset $B \subseteq \{1,2, \ldots , h\}$ of size $h-1$ that is matchable in $\Gamma_G$. By Corollary~\ref{cor:BP} $A$ is matchable, hence Theorem~\ref{thm:MD} provides a matching $M$ in $\Gamma_G$ with $A \cup B \subseteq \bigcup M$. Let $I \subseteq \{1,2, \ldots , r\}$ be the set of vertices that $M$ covers on this side, and $I^c = \{1,2, \ldots , r\} \setminus I$. As $B \subseteq I$, we know that $I^c$ contains at most one $1 \le i \le h$ (actually exactly one such $i$, because $\{1,2, \ldots , h\}$ is not matchable by Claim~\ref{cl:No}). Because we assume $r - |A| \ge 3$, there are at least $2$ other members of $I^c$, and they are in $\{h+1, \ldots , r\}$.

Returning from $\Gamma_G$ to $G$, we can replace each edge $iu$ in $M$ by an edge $v_iu$ of $G$, where $v_i \in D_i$. Let $M'$ be the obtained matching in $G$. By adding to $M'$ a perfect matching of $G[D_i] - \{v_i\}$ for each $i \in I$, and a perfect matching of $G[C]$, we get a matching $M''$ in $G$ that covers $V \setminus \bigcup_{i \in I^c} D_i$. For the unique $D_i$ with $i \in I^c$, $1 \le i \le h$, we take $\frac{|D_i|+1}{2}$ edges of $G$ that cover it. All these edges, as well as those of $M''$, can be replaced by edges of $H$ containing them. There remain at least $2$ $D_i$'s to be covered, having $h < i \le r$. By Claim~\ref{cl:Low} each of them can be covered by at most $\frac{|D_i|-1}{2}$ edges of $H$. Altogether we get an edge cover in $H$ using at most $\frac{|V|-1}{2}$ edges, because we had just one $\frac{|D_i|+1}{2}$ term and at least two $\frac{|D_i|-1}{2}$ terms. This contradicts our assumption that $\rho(H) > \frac{|V|-1}{2}$.
\end{proof}
Claim~\ref{cl:Ar} confirms that the Gallai-Edmonds decomposition of $G$ has the form indicated in (c1): $|A| = r - 1$, and we already observed that $|A| \ge 1$. It remains to prove the existence of an edge $e_k \in E_2 \setminus E(G)$ satisfying (c3), that we can add to $G$ to get the next graph in our sequence. This is shown in the next, final claim.
\begin{claim} \label{cl:Ek}
Let $S \subseteq \{1,2, \ldots , h\}$ be a subset as guaranteed to exist in Claim~\ref{cl:One}. Then there exists an $i \in S$ such that $x \notin D_i$, and there is an edge $e_k$ in $E_2 \setminus E(G)$ that has exactly one endpoint in $D_i$.
\end{claim}
\begin{proof}
Let $|S| = s \ge 2$. The vertex $x$ belongs to at most one $D_j$, $j \in S$. Let $j$ be that member of $S$ if it exists, or an arbitrary member otherwise. Then $|S \setminus \{j\}| = s-1 \ge 1$, and for ease of notation we assume that $S \setminus \{j\} = \{1,2, \ldots , s-1\}$. Now consider the following $s$ disjoint subsets of $V$: \[ D_1, D_2, \ldots , D_{s-1}, V \setminus \bigl( \bigcup _{i=1}^{s-1} D_i \cup N(S) \bigr).\] First, observe that each of these sets has odd cardinality. To check this for the last one, note that $|N(S)| = s-1$. Thus, the parity of $\bigcup _{i=1}^{s-1} D_i \cup N(S)$ is that of $(s-1) + (s-1)$, hence even, while $|V|$ is odd. Next, note that $G$ has no edges between any two of these $s$ subsets. Indeed, any outside neighbors of vertices in one of the $D_i$'s lie in $N(S)$.

Now, take away some $u \in N(S)$. As $G_H$ is factor-critical, we can find a perfect matching $M$ in $G_H - \{u\}$. The matching $M$ covers all the vertices in our $s$ odd subsets, and the remaining $s-2$ vertices of $N(S) \setminus \{u\}$. Necessarily $M$ has at least one edge $e_k$ connecting two of our $s$ subsets. Such $e_k$ cannot be an edge of $G$, and in particular is not in $\partial E_3$, hence $e_k \in E_2 \setminus E(G)$. Of the two sets that $e_k$ connects, at least one is a $D_i$, and such $D_i$ satisfies our claim together with the edge $e_k$.
\end{proof}

Note that the statement of Claim~\ref{cl:Ek} suffices to satisfy the requirements of (c3), because $|D_i| \ge 3$ for all $i \in S$ by Claim~\ref{cl:One}.
\end{proof}

\subsection{No such structured counterexample to Proposition~\ref{pro:Cover} exists}
We begin this subsection with some terminology and notation. An \emph{edge-ordered graph} $G_<$ is a graph $G = (V, E)$ equipped with a linear order $<$ on its edge set $E$. If $e_1 < e_2$ we think of $e_1$ as having appeared before $e_2$. For $e \in E$, we denote by $G^{(e)} = (V, E^{(e)})$ the subgraph of $G$ formed before $e$, i.e., $E^{(e)} = \{ e' \in E: e' < e\}$.

Let $V$ be a vertex set and let $G$ and $H$ be a graph and a $3$-uniform hypergraph, respectively, on $V$. A family of disjoint edges, one of which is in $E(H)$ and the others in $E(G)$, whose union is $U \subseteq V$ will be called a $G!H$ \emph{cover} of $U$. Note that the number of edges in a $G!H$ cover of $U$ is $\frac{|U|-1}{2}$. We will say that $U$ is $G!H$ \emph{coverable} if such a family exists.

The following lemma will serve to prove that a counterexample to Proposition~\ref{pro:Cover}, structured as in Lemma~\ref{lem:Cons}, cannot exist.
\begin{lemma} \label{lem:G!H}
Let $V$ be a vertex set, $|V| > 1$. Let $G_<$ be a factor-critical edge-ordered graph on $V$, and let $H$ be a $3$-uniform hypergraph on $V$. Assume that every edge of $H$ contains at least two edges of $G$. Suppose that there exists a vertex $x \in V$ so that every edge $e \in E(G)$ satisfies one of the following conditions:
\begin{itemize}
    \item[($\alpha$)] $x \in e$.
    \item[($\beta$)] $e \in \partial E(H)$.
    \item[($\gamma$)] The Gallai-Edmonds decomposition of $G^{(e)}$ is of the form $D^{(e)} = D_1^{(e)} \cup D_2^{(e)} \cup \cdots \cup D_{r^{(e)}}^{(e)}$, $A^{(e)}$, $C^{(e)}$ where $|A^{(e)}| = r^{(e)} - 1 \ge 1$, and $e$ has one endpoint in $D_{j^{(e)}}^{(e)}$ and the other outside it, where $1 \le j^{(e)} \le r^{(e)}$ is such that $|D_{j^{(e)}}^{(e)}| \ge 3$ and $x \notin D_{j^{(e)}}^{(e)}$.
\end{itemize}
Then $V$ is $G!H$ coverable.
\end{lemma}
We first show how Lemma~\ref{lem:Cons} and Lemma~\ref{lem:G!H} combine to prove Proposition~\ref{pro:Cover}.

\vspace{10pt}

\noindent {\it Proof of  Proposition~\ref{pro:Cover} from Lemma~\ref{lem:Cons} and Lemma~\ref{lem:G!H}.} \,Assume, for the sake of contradiction, that Proposition~\ref{pro:Cover} is false. Then it has a minimal counterexample $H = (V, E)$. By Lemma~\ref{lem:Cons} there exists a vertex $x \in V$ and there is a sequence $G^{(1)}, G^{(2)}, \ldots , G^{(t)}$ of subgraphs of $G_H$ satisfying (a), (b) and (c) in that lemma.

We are going to apply Lemma~\ref{lem:G!H} to the graph $G^{(t)}$, which is factor-critical by (b), equipped with a suitable order $<$ on $E^{(t)}$, and the hypergraph $(V, E_3)$. The order $<$ corresponds to the construction of $G^{(t)}$: the edges in (a), namely $E^{(1)} = \{e \in E_2: x \in e\} \cup \partial E_3$ come first (in arbitrary order) and then, if $t > 1$, come the edges $e_1, e_2, \ldots , e_{t-1}$ added in (c), in this order. Every edge in $E_3$ contains three edges of $\partial E_3$, hence of $G^{(t)}$. The fact that, for the vertex $x$ given by Lemma~\ref{lem:Cons}, every edge $e \in E^{(t)}$ satisfies condition ($\alpha$), ($\beta$) or ($\gamma$) in Lemma~\ref{lem:G!H} follows from the construction. Indeed, if $e$ fails to satisfy ($\alpha$) or ($\beta$), then $e = e_k$ for some $1 \le k < t$. But then the graph $G^{(e)}$ in the notation of Lemma~\ref{lem:G!H} is exactly $G^{(k)}$ in the construction of Lemma~\ref{lem:Cons}, and condition ($\gamma$) for $e$ follows from condition (c) of the construction.

Thus, all the requirements in Lemma~\ref{lem:G!H} hold. The lemma guarantees the existence of a family of $\frac{|V|-1}{2}$ edges covering $V$, one of which is in $E_3$ and the others in $E^{(t)}$. Every edge in $E^{(t)}$ is in $E(G_H)$, hence contained in an edge of $H$. Replacing those edges by edges of $H$ containing them, we obtain an edge cover in $H$ of size $\frac{|V|-1}{2}$. This contradicts our choice of $H$ as a counterexample to Proposition~\ref{pro:Cover}. \qed

\vspace{10pt}

It remains to prove Lemma~\ref{lem:G!H}. The proof will be by induction on $|V|$, using contraction in the induction step. The following claim will be useful in this context, and we start with its statement and proof.
\begin{claim} \label{cl:Out}
Let $G$ and $H$ be a graph and a $3$-uniform hypergraph, respectively, on the same vertex set $V$. Assume that $G$ is factor-critical, and every edge of $H$ contains at least two edges of $G$. Let $O$ be an induced odd cycle in $G$ such that $V(O)$ is not $G!H$ coverable and $E(O) \subseteq \partial E(H)$. Suppose that $u_1v_1 \in E(G)$, $u_1 \notin V(O)$, $v_1 \in V(O)$. Then there exists $u_2v_2 \in E(G)$ such that $u_2 \notin V(O)$, $v_2 \in V(O)$, $u_2 \ne u_1$, and $V(O) \cup \{u_1,u_2\}$ is $G!H$ coverable.
\end{claim}
\begin{proof} By assumption, for every edge $xy$ of $O$ there exists a vertex $z$ such that $xyz$ is an edge of $H$. We argue that $z$ cannot be a vertex of $O$. Indeed, if it is, then we consider two cases. If $x, y, z$ are consecutive on the cycle $O$ then using $xyz$ and a perfect matching of the rest of the cycle we get a $G!H$ cover of $V(O)$, contradicting our assumption. If $x, y, z$ are not consecutive then, by another assumption, either $xz$ or $yz$ is an edge of $G$ which is a chord of $O$, contradicting the fact that $O$ is induced. We conclude that for every edge $xy$ of $O$ there exists a vertex $z = z(xy) \notin V(O)$ such that $xyz$ is an edge of $H$.

Now, let $u_1v_1 \in E(G)$ with $u_1 \notin V(O)$, $v_1 \in V(O)$ be given as in the claim. In order to find an edge $u_2v_2 \in E(G)$ as required, we distinguish between three cases.

\vspace{5pt}

\noindent \underline{Case I}. $z(xy) = u_1$ for all $xy \in E(O)$.

As $G$ is factor-critical, there exists a perfect matching $M$ in $G - \{u_1\}$. Since $|V(O)|$ is odd, $M$ must contain an edge $u_2v_2$ with $u_2 \notin V(O)$, $v_2 \in V(O)$, $u_2 \ne u_1$. Let $xy$ be an edge of $O$ so that $x, y, v_2$ are consecutive. We get a $G!H$ cover of $V(O) \cup \{u_1,u_2\}$ by taking $xyu_1$, $u_2v_2$ and a perfect matching of $O - \{x, y, v_2\}$.

\vspace{5pt}

\noindent \underline{Case II}. $z(xy) \ne u_1$ for all $xy \in E(O)$.

Let $xy$ be an edge of $O$ so that $x, y, v_1$ are consecutive. Let $z(xy) = u_2 \ne u_1$. As $xyu_2$ is an edge of $H$, one of $xu_2$, $yu_2$ is an edge of $G$. We choose $v_2$ to be $x$ or $y$ accordingly. We get a $G!H$ cover of $V(O) \cup \{u_1, u_2\}$ by taking $xyu_2$, $u_1v_1$ and a perfect matching of $O - \{x, y, v_1\}$.

\vspace{5pt}

\noindent \underline{Case III}. Both $z(xy) = u_1$ and $z(xy) \ne u_1$ occur among edges $xy \in E(O)$.

In this case, we can find such occurrences at two adjacent edges of the cycle. Say $x, y, v$ are consecutive, $z(xy) = u_1$ and $z(yv) = u_2 \ne u_1$. If $xu_1 \in E(G)$ then we take $v_2$ to be $y$ or $v$, whichever of them satisfies $u_2v_2 \in E(G)$. We get a $G!H$ cover of $V(O) \cup \{u_1, u_2\}$ using $yvu_2$, $xu_1$ and a perfect matching of $O - \{x, y, v\}$. If $vu_2 \in E(G)$ then we take $v_2 = v$, and use $xyu_1$, $vu_2$ and a perfect matching of $O - \{x, y, v\}$ for a $G!H$ cover of $V(O) \cup \{u_1, u_2\}$.

The remaining possibility is that none of $xu_1$ and $vu_2$ is in $E(G)$. Then we know that $yu_1$ and $yu_2$ are both in $E(G)$. Now we consider the edge $vw$ of the cycle, where $w \ne y$ (possibly $w = x$ when $O$ is a triangle), and distinguish between two cases. If $z(vw) = u_1$ we take $v_2 = y$, and use $vwu_1$, $yu_2$ and a perfect matching of $O - \{y, v, w\}$ for a $G!H$ cover of $V(O) \cup \{u_1, u_2\}$. Finally, if $z(vw) = u_3 \ne u_1$ then we let $u_3$ play the role of $u_2$ in the claim, and $v$ or $w$ play the role of $v_2$. We get a $G!H$ cover of $V(O) \cup \{u_1, u_3\}$ by taking $vwu_3$, $yu_1$ and a perfect matching of $O - \{y, v, w\}$.
\end{proof}

\vspace{10pt}

\noindent {\it Proof of Lemma~\ref{lem:G!H}.} \,Let $V$, $G_<$, $H$, and $x \in V$ be as in the statement of the lemma. As $G$ is factor-critical and $|V| > 1$, let $U_1, U_2, \ldots , U_{2k+1}$ be a partition of $V$ as in Lemma~\ref{lem:Odd}. We first treat the special case in which both of the following hold:
\begin{itemize}
    \item Every edge $e \in E(G)$ satisfies condition ($\alpha$) or ($\beta$) in Lemma~\ref{lem:G!H}.
    \item Every $U_i$, except possibly the one containing $x$, is a singleton.
\end{itemize}
In this special case, we prove the lemma by explicitly constructing a $G!H$ cover of $V$.

Assume w.l.o.g.\ that $x \in U_1$. Let $U_i = \{v_i\}$ for $i=2,3, \ldots , 2k+1$. Consider the edge $v_2v_3 \in E(G)$. Because it does not contain $x$, it must satisfy condition ($\beta$). Thus, there exists a vertex $u$ such that $v_2v_3u \in E(H)$. Suppose first that $u = v_i$ for some $5 \le i \le 2k+1$. Since $v_2v_3u$ must contain at least two edges of $G$, either $v_2u$ or $v_3u$ is in $E(G)$. In either case, we get an edge between two non-cyclically adjacent $U_i$'s, contradicting the structure described in Lemma~\ref{lem:Odd}. It follows that $u = v_4$ or $u \in U_1$.

If $u = v_4$ we obtain a $G!H$ cover of $V$ by taking $v_2v_3u$, $v_5v_6, \ldots , v_{2k-1}v_{2k}$, $v_{2k+1}v$ for some $v \in U_1$, and a perfect matching in $G[U_1] - \{v\}$. If $u \in U_1$ we use $v_2v_3u$, $v_4v_5, \ldots , v_{2k}v_{2k+1}$, and a perfect matching in $G[U_1] - \{u\}$. This proves the lemma in the special case.

When the special case does not hold, there are two possibilities.

\vspace{5pt}

\noindent \underline{Possibility 1}. Every edge $e \in E(G)$ satisfies condition ($\alpha$) or ($\beta$) in Lemma~\ref{lem:G!H}, and there exists a $U_i$ such that $x \notin U_i$ and $|U_i| > 1$.

\vspace{5pt}

\noindent \underline{Possibility 2}. There exists an edge $e \in E(G)$ that satisfies only condition ($\gamma$) in Lemma~\ref{lem:G!H}.

\vspace{5pt}

We will handle these possibilities by induction on $|V|$. We begin with a general description of the argument, which is valid for both possibilities. Later on, we will fill in the details, some of which differ between the two possibilities.

We assume, for the sake of contradiction, that $V$ is not $G!H$ coverable. We choose an induced odd cycle $O$ in $G$ so that the following conditions hold:
\begin{equation} \label{eq:noxo}
x \notin V(O).
\end{equation}
\begin{equation} \label{eq:osh}
E(O) \subseteq \partial E(H).
\end{equation}
\begin{equation} \label{eq:cfc}
G/V(O) \textrm{ is factor-critical}.
\end{equation}
The particular way we choose such $O$ will be explained later. For ease of notation, we write $G' = (V', E')$ for the contracted graph $G/V(O)$, and denote the vertex formed in the contraction by $v_O \in V'$.

We also define a $3$-uniform hypergraph $H'$ on $V'$. It coincides with $H$ on $V' \setminus \{v_O\}$, and the edges containing $v_O$ are those $\{u_1, u_2, v_O\}$ satisfying:
\begin{itemize}
    \item[-] $V(O) \cup \{u_1, u_2\}$ is $G!H$ coverable.
    \item[-] $\{ u_1, u_2, v_O\}$ contains at least two edges of $G'$.
\end{itemize}

Suppose for the moment that $V'$, $G'_<$ (with a suitable order $<$ on $E'$), $H'$, and $x \in V'$ satisfy the assumptions of Lemma~\ref{lem:G!H}. Since $|V'| < |V|$ (because the contraction collapses at least three vertices into one), the induction hypothesis provides a $G'!H'$ cover of $V'$. From it, we can construct a $G!H$ cover of $V$ as follows. If $v_O$ belongs to the $H'$ edge in the cover of $V'$, we replace that edge $u_1u_2v_O$ by a $G!H$ cover of $V(O) \cup \{u_1, u_2\}$. If $v_O$ belongs to a $G'$ edge in the cover of $V'$, say $uv_O$, we replace it by $uv$ for some $v \in V(O)$ and add a perfect matching in $O - \{v\}$. In either case, $V$ turns out to be $G!H$ coverable and we are done.

Thus, it remains to indicate our choice of an induced odd cycle $O$ in $G$ so that (\ref{eq:noxo}), (\ref{eq:osh}) and (\ref{eq:cfc}) are satisfied, and to verify that under this choice, with a suitable order on $E'$, the assumptions of Lemma~\ref{lem:G!H} hold for $V'$, $G'_<$, $H'$, and $x \in V'$. We start with verifying the assumptions, which for the most part depends only on (\ref{eq:noxo}), (\ref{eq:osh}) and (\ref{eq:cfc}).

First, $|V'| > 1$ because $\{x, v_O\} \subseteq V'$. $G'$ is factor-critical by (\ref{eq:cfc}), and $H'$ is $3$-uniform by definition. The property that an edge of $H'$ contains at least two edges of $G'$ is inherited from the same property of $H$ with respect to $G$ when the edge of $H'$ does not contain $v_O$; and it is part of the definition of $H'$ when the edge of $H'$ contains $v_O$.

For an edge $e$ of $G$ that is not contained in $V(O)$, let us denote by $e'$ the corresponding edge of $G'$. That is, $e' = e$ if $e$ is disjoint from $V(O)$, and $e' = uv_O$ if $e = uv$ for some $u \notin V(O)$, $v \in V(O)$.

Using the fact that $e \in E(G)$ satisfies ($\alpha$), ($\beta$) or ($\gamma$) in the statement of the lemma, we would like to show that the same is true for $e' \in E(G')$ with respect to $H'$. Regarding condition ($\alpha$), this is immediate:
\begin{equation} \label{eq:con1}
x \in e \,\,\Longrightarrow\,\, x \in e'.
\end{equation}
Indeed, by (\ref{eq:noxo}) $x \notin V(O)$, so this follows from the definition of the contraction. For condition ($\beta$) the argument is more involved, and we start by showing the following:
\begin{equation} \label{eq:shad}
e = uv,\, u \notin V(O),\, v \in V(O)\,\,\Longrightarrow\,\, e' \in \partial E(H').
\end{equation}
Note that for such $e$, we assert that $e'$ satisfies ($\beta$) regardless of whether or not $e$ satisfies it. The proof is an application of Claim~\ref{cl:Out}. First, we verify that $G$, $H$, and $O$ satisfy all the assumptions of Claim~\ref{cl:Out}. The only assumption that requires some checking is that $V(O)$ is not $G!H$ coverable. Suppose it is. Because $G'$ is factor-critical by (\ref{eq:cfc}), we know that $G - V(O)$ has a perfect matching $M$. Adding $M$ to a $G!H$ cover of $V(O)$ gives a $G!H$ cover of $V$, contradicting our assumption that $V$ is not $G!H$ coverable.

Thus we can apply Claim~\ref{cl:Out} with the given edge $e = uv$ in the role of $u_1v_1$. We get another edge $u_2v_2 \in E(G)$ such that $u_2 \notin V(O)$, $v_2 \in V(O)$, $u_2 \ne u_1$, and $V(O) \cup \{u_1, u_2\}$ is $G!H$ coverable. By the definition of $H'$ we obtain that $u_1u_2v_O \in E(H')$; indeed, the second requirement in the definition holds because both $u_1v_O$ and $u_2v_O$ are edges of $G'$ obtained by contracting $u_1v_1$ and $u_2v_2$, respectively. Thus $e' = uv_O$, being contained in $u_1u_2v_O$ (recall that $u = u_1$), satisfies $e' \in \partial E(H')$, which proves (\ref{eq:shad}).

Next, we prove:
\begin{equation} \label{eq:con2}
e \nsubseteq V(O),\, e \in \partial E(H)\,\,\Longrightarrow\,\, e' \in \partial E(H').
\end{equation}
In view of (\ref{eq:shad}), it remains to prove this implication for edges $e$ disjoint from $V(O)$. Then we have $e = u_1u_2$ with $u_1, u_2 \notin V(O)$, and $e' = e$. By assumption, there exists a vertex $y \in V$ such that $u_1u_2y \in E(H)$. If $y \notin V(O)$ then $u_1u_2y \in E(H')$ and we are done. If $y \in V(O)$ then we claim that $u_1u_2v_O \in E(H')$. Indeed, $V(O) \cup \{u_1, u_2\}$ is $G!H$ coverable, by taking $u_1u_2y$ and a perfect matching in $O - \{y\}$. Moreover, since $u_1u_2y$ contains at least two edges of $G$, $u_1u_2v_O$ contains at least two edges of $G'$. Thus, by the definition of $H'$, $u_1u_2v_O \in E(H')$. As it contains $e' = u_1u_2$, this confirms that $e' \in \partial E(H')$.

So far we have seen that if $e$ satisfies condition ($\alpha$), respectively ($\beta$), then so does $e'$. We leave condition ($\gamma$) for later, and proceed to handle Possibility 1. In this case, every edge of $G$ satisfies ($\alpha$) or ($\beta$), so by the above every edge of $G'$ satisfies ($\alpha$) or ($\beta$). Thus, the assumptions of Lemma~\ref{lem:G!H} hold for $V'$, $G'_<$, $H'$, and $x \in V'$ regardless of the order $<$ (which is only relevant to condition ($\gamma$)). To complete the proof in this case, we just show that an induced odd cycle $O$ in $G$ can be chosen so that (\ref{eq:noxo}), (\ref{eq:osh}) and (\ref{eq:cfc}) are satisfied.

By the assumption of Possibility 1, we can find a part $U_i$ in the partition of $V$ provided by Lemma~\ref{lem:Odd} so that $x \notin U_i$ and $|U_i| > 1$. As $G[U_i]$ is factor-critical, Lemma~\ref{lem:Ind} guarantees the existence of an induced odd cycle $O$ in $G[U_i]$ such that $G[U_i]/V(O)$ is factor-critical. We claim that such $O$ satisfies (\ref{eq:noxo}), (\ref{eq:osh}) and (\ref{eq:cfc}). Indeed, (\ref{eq:noxo}) holds by the choice of $U_i$. Since every edge of $G$ satisfies ($\alpha$) or ($\beta$), but the edges of $O$ do not satisfy ($\alpha$), all of them satisfy ($\beta$) and (\ref{eq:osh}) holds. To verify (\ref{eq:cfc}), it suffices to check that $G/V(O)$ admits a partition as described in Lemma~\ref{lem:Odd}. Indeed, taking the partition $U_1, U_2, \ldots , U_{2k+1}$ corresponding to $G$ and replacing the $U_i$ that contains $V(O)$ by $\bigl( U_i \setminus V(O) \bigr) \cup \{v_O\}$ yields such a partition for $G/V(O)$. This completes the inductive proof of the lemma in the case when Possibility 1 holds.

From now on, we assume that Possibility 2 holds. Thus, there exists an edge $e \in E(G)$ that satisfies only condition ($\gamma$) in the lemma. We denote by $\bar{e}$ the first such edge according to the order $<$ on $E(G)$, and consider the Gallai-Edmonds decomposition of $G^{(\bar{e})}$, namely $D^{(\bar{e})} = D_1^{(\bar{e})} \cup D_2^{(\bar{e})} \cup \cdots \cup D_{r^{(\bar{e})}}^{(\bar{e})}$, $A^{(\bar{e})}$, $C^{(\bar{e})}$. By condition ($\gamma$), we know that $|A^{(\bar{e})}| = r^{(\bar{e})} - 1 \ge 1$, and there is a component $D_{j^{(\bar{e})}}^{(\bar{e})}$ such that $|D_{j^{(\bar{e})}}^{(\bar{e})}| \ge 3$, $x \notin D_{j^{(\bar{e})}}^{(\bar{e})}$, and $\bar{e}$ has one endpoint in $D_{j^{(\bar{e})}}^{(\bar{e})}$ and the other outside it.

We first prove:
\begin{equation} \label{eq:stay}
\bar{e} \le e \,\,\Longrightarrow\,\, D_{j^{(\bar{e})}}^{(\bar{e})} \subseteq D_k^{(e)} \textrm{ for some } 1 \le k \le r^{(e)}.
\end{equation}
Indeed, as $|A^{(\bar{e})}| = r^{(\bar{e})} - 1$, for every vertex in $D^{(\bar{e})}$ there is a near-perfect matching in $G^{(\bar{e})}$, hence also in $G^{(e)}$, avoiding that vertex. This implies that $D^{(\bar{e})} \subseteq D^{(e)}$, and because vertices in the same component of $D^{(\bar{e})}$ are still connected in $G^{(e)}$, they must lie in the same component of $D^{(e)}$, proving (\ref{eq:stay}).

As $G^{(\bar{e})}[D_{j^{(\bar{e})}}^{(\bar{e})}]$ is factor-critical, we can find by Lemma~\ref{lem:Ind} an induced odd cycle $\bar{O}$ in it such that $G^{(\bar{e})}[D_{j^{(\bar{e})}}^{(\bar{e})}]/V(\bar{O})$ is factor-critical. Note that $\bar{O}$ is induced in $G^{(\bar{e})}$ but is not necessarily induced in $G$. If it is, we take $O = \bar{O}$ in our proof. If it is not, then using Lemma~\ref{lem:Short}, we take $O$ to be an odd shortening of $\bar{O}$ which is induced in $G$. We proceed to verify that $O$ has the required properties.

Since $x \notin D_{j^{(\bar{e})}}^{(\bar{e})}$, (\ref{eq:noxo}) holds. To check (\ref{eq:osh}), we note that the edges of $O$ are either edges of $\bar{O}$ or chords of $\bar{O}$. In the former case, as $\bar{O}$ lies in $G^{(\bar{e})}$, they must satisfy ($\alpha$) or ($\beta$) by the minimality of $\bar{e}$. But ($\alpha$) is already excluded, so these edges satisfy ($\beta$): they are in $\partial E(H)$. In the latter case, if $e$ is a chord of $\bar{O}$ then $e \in E(G) \setminus E(G^{(\bar{e})})$, hence $\bar{e} \le e$. It follows from (\ref{eq:stay}) that $e \subseteq D_k^{(e)}$ for some $1 \le k \le r^{(e)}$. Thus $e$ does not satisfy condition ($\gamma$), and because ($\alpha$) is excluded, we deduce that $e$ satisfies ($\beta$). This concludes the verification of (\ref{eq:osh}).

In order to check (\ref{eq:cfc}), we first show:
\begin{equation} \label{eq:pm}
G^{(\bar{e})}[D_{j^{(\bar{e})}}^{(\bar{e})} \setminus V(O)] \textrm{ has a perfect matching}.
\end{equation}
Indeed, because $G^{(\bar{e})}[D_{j^{(\bar{e})}}^{(\bar{e})}]/V(\bar{O})$ is factor-critical, there is a perfect matching of $D_{j^{(\bar{e})}}^{(\bar{e})} \setminus V(\bar{O})$ in $G^{(\bar{e})}$. If $O = \bar{O}$ this is exactly what we need. If $O$ is an odd shortening of $\bar{O}$ formed upon replacing the odd paths $Q_1, Q_2, \ldots , Q_{\ell}$ by chords, we just need to add perfect matchings of the internal vertices of these odd paths. This proves (\ref{eq:pm}).

To verify (\ref{eq:cfc}), it suffices by Lemma~\ref{lem:PM} to show that $G - V(O)$ has a perfect matching. We show that already $G^{(\bar{e})} - V(O)$ has one, and do this in two parts: $V \setminus D_{j^{(\bar{e})}}^{(\bar{e})}$ has a perfect matching in $G^{(\bar{e})}$ due to the Gallai-Edmonds decomposition having $|A^{(\bar{e})}| = r^{(\bar{e})} - 1$, and $D_{j^{(\bar{e})}}^{(\bar{e})} \setminus V(O)$ has a perfect matching in $G^{(\bar{e})}$ by (\ref{eq:pm}). This completes the proof that $O$ satisfies (\ref{eq:noxo}), (\ref{eq:osh}) and (\ref{eq:cfc}).

We now define an order on $E' = E\bigl(G/V(O)\bigr)$. For this purpose, we write each edge in $E'$ as $e'$, where $e$ is the first edge of $G$ whose contraction gives this $e'$. We use the given order on $E(G)$ to order $E'$: \[ e'_1 < e'_2\,\, \Longleftrightarrow\,\,e_1 < e_2.\]
Note that under this definition of the order on $E'$, contraction and taking the subgraph of preceding edges commute: $G'^{(e')} = G^{(e)}/V(O)$.

It remains to complete the verification that the assumptions of the lemma hold for $V'$, $G'_<$, $H'$, and $x \in V'$; this will allow the inductive argument to take care of Possibility 2. Every edge $e$ of $G$ is known to satisfy ($\alpha$), ($\beta$) or ($\gamma$), and we have already shown in (\ref{eq:con1}) and (\ref{eq:con2}) that if it satisfies ($\alpha$) or ($\beta$) then so does $e'$. Thus, we may assume that $e \in E(G)$ satisfies only condition ($\gamma$). By the choice of $\bar{e}$, we have $\bar{e} \le e$. If $e$ has one vertex in $V(O)$ and the other outside it, then by (\ref{eq:shad}) $e'$ satisfies ($\beta$). Thus, we may assume that $e \cap V(O) = \emptyset$ and therefore $e' = e$.

We will show that in this case $e'$ satisfies ($\gamma$). For this we need to describe the Gallai-Edmonds decomposition of $G'^{(e')} = G^{(e)}/V(O)$. As $\bar{e} \le e$, we know from (\ref{eq:stay}) that $V(O) \subseteq D_{j^{(\bar{e})}}^{(\bar{e})} \subseteq D_k^{(e)}$ for some $1 \le k \le r^{(e)}$. We claim that the Gallai-Edmonds decomposition of $G^{(e)}/V(O)$ is obtained from that of $G^{(e)}$ upon replacing $D_k^{(e)}$ by $\bigl( D_k^{(e)} \setminus V(O)\bigr) \cup \{v_O\}$. To check this, we only need to verify that $G^{(e)}[D_k^{(e)}]/V(O)$ is factor-critical (the other structural requirements are not affected by the contraction). As $G^{(e)}[D_k^{(e)}]$ is factor-critical, by Lemma~\ref{lem:PM} it suffices to show that $D_k^{(e)} \setminus V(O)$ has a perfect matching in $G^{(e)}$. In fact, we show below that it has a perfect matching already in $G^{(\bar{e})}$, which is a subgraph of $G^{(e)}$.

Indeed, let $M$ be a near-perfect matching in $G^{(\bar{e})}$ that misses a vertex $v \in D_{j^{(\bar{e})}}^{(\bar{e})}$. By Corollary~\ref{cor:MM}, $M$ has a subset $M'$ with $\bigcup M' = D_{j^{(\bar{e})}}^{(\bar{e})} \setminus \{v\}$. Another application of this corollary, when viewing $M$ as a near-perfect matching in $G^{(e)}$, shows that $M$ has a subset $M''$ with $\bigcup M'' = D_k^{(e)} \setminus \{v\}$. Thus $\bigcup (M'' \setminus M') = D_k^{(e)} \setminus D_{j^{(\bar{e})}}^{(\bar{e})}$. Adding to $M'' \setminus M'$ a perfect matching in $G^{(\bar{e})}[D_{j^{(\bar{e})}}^{(\bar{e})} \setminus V(O)]$ which exists by (\ref{eq:pm}), we obtain a perfect matching of $D_k^{(e)} \setminus V(O)$ in $G^{(\bar{e})}$, as promised. This confirms that $G^{(e)}[D_k^{(e)}]/V(O)$ is factor-critical, and therefore the Gallai-Edmonds decomposition of $G'^{(e')}$ is as described above.

We know that $e$ satisfies condition ($\gamma$) with respect to the decomposition of $G^{(e)}$, and $e' = e$. In order to verify that $e'$ satisfies ($\gamma$) with respect to the decomposition of $G'^{(e')}$, we only need to check that after contracting $D_k^{(e)}$ to $\bigl( D_k^{(e)} \setminus V(O)\bigr) \cup \{v_O\}$, the $j(e)$ component still has more than one vertex. If $j(e) \ne k$ there is nothing to prove. If $j(e) = k$ then, after contraction, this component still contains one vertex of $e$ (which is disjoint from $V(O)$), as well as $v_O$. So indeed $e'$ satisfies condition ($\gamma$).

This completes the argument that under Possibility 2, every edge in $E(G')$ satisfies ($\alpha$), ($\beta$) or ($\gamma$). Thus, the inductive proof of the lemma works for both possibilities. This ends the proof of Lemma~\ref{lem:G!H}, hence also of Proposition~\ref{pro:Cover} and Theorem~\ref{thm:3/2}. \qed

\section{A Gale-Shapley type theorem for triples}
The well-known stable matching theorem of Gale and Shapley~\cite{GS} can be formulated in terms of bipartite graphs as follows. Consider a bipartite graph $G = (V, E)$, where every vertex $v$ has a linear order $\succeq_v$ on the set $E_v$ of edges incident to $v$. Here $vx \succeq_v vy$ means that $v$ prefers to be matched to $x$ than to $y$ (or $x=y$). A matching $M$ in $G$ dominates an edge $e = uv \in E$ if either $u$ or $v$ has an edge incident to it in $M$ that it weakly prefers to $e$. A matching is stable if it dominates every edge in $E$. The stable matching theorem says that every bipartite graph, with any system of preferences as above, has a stable matching. This theorem and related ones have been much studied and widely applied in practice, e.g., to job markets, school assignments, and matching of organ donors to patients.

One direction of study is to allow the sets that can be formed (the edges) to be larger, not just couples. This is modeled by a hypergraph $H = (V, E)$ and a \emph{preference system} $\bigl( \succeq_v\bigr)_{v \in V}$, where each $\succeq_v$ is a linear order on $E_v$, the set of edges in $E$ that contain $v$. A set of edges $\{e_1, e_2, \ldots , e_m\} \subseteq E$ \emph{dominates} an edge $e \in E$ at a vertex $v \in e$ if for some $1 \le i \le m$ we have $v \in e_i$ and $e_i \succeq_v e$. We say that $\{e_1, e_2, \ldots , e_m\}$ is \emph{stable} if it dominates every edge $e \in E$ at some $v \in e$.

When $H$ is a bipartite graph, this coincides with the standard model described above, and a stable matching is guaranteed to exist. Beyond bipartite graphs, however, this is no longer true. Simple examples without a stable matching are known already for non-bipartite graphs (``the roommates problem"~\cite{GS}) and for $3$-partite hypergraphs (``threesome matchings"~\cite{Al}).

One way to regain existence is to relax the notion of matching to its fractional version. A \emph{fractional matching} in $H = (V, E)$ is a function $f: E \to \mathbb{R}_+$ that satisfies $\sum_{e \ni v} f(e) \le 1$ for every vertex $v \in V$. Note that these inequalities are opposite to those required for a fractional edge cover. We say that a vertex $v$ is \emph{saturated} by $f$ if $\sum_{e \ni v} f(e) = 1$. Given a preference system for $H$, a fractional matching $f$ in $H$ \emph{dominates} an edge $e \in E$ at a vertex $v \in e$ if \[ \sum \{ f(e'): e' \in E_v, e' \succeq_v e\} = 1.\] Observe that this requires $v$ to be saturated by $f$, and all edges $e' \in E_v$ with $f(e') > 0$ to be weakly preferred to $e$ by $v$. We say that $f$ is \emph{stable} if it dominates every edge $e \in E$ at some $v \in e$. Note that if $f$ is $\{0,1\}$-valued, this means that the $1$-valued edges form a stable matching. The following theorem of Aharoni and Fleiner~\cite{AF} guarantees the existence of a stable fractional matching (in the special case of graphs, this is implicit in an earlier result of Tan~\cite{Tan}).
\begin{theorem}[{\cite{AF}}] \label{thm:AF}
For any hypergraph $H = (V, E)$ and preference system $\bigl( \succeq_v\bigr)_{v \in V}$ there exists a stable fractional matching.
\end{theorem}
Here we propose another relaxation of the notion of stable matching. Instead of looking at matchings, we consider arbitrary sets of edges, and want to guarantee the existence of a stable one whose size is small (in terms of the number of vertices and the edge size). We will obtain such a result for $3$-uniform hypergraphs by combining Theorem~\ref{thm:AF} and Corollary~\ref{cor:Ex}.

More precisely, we need a slight generalization of Corollary~\ref{cor:Ex}, which may be of independent interest. In the corollary, we assumed the existence of an exact fractional edge cover, which is the same as a fractional matching which saturates all the vertices. The following proposition is an adaptation of the corollary that handles fractional matchings which may not saturate all the vertices.
\begin{proposition} \label{pro:SAT}
Let $H = (V, E)$ be a $3$-uniform hypergraph with $|V| = n$. Let $f: E \to \mathbb{R}_+$ be a fractional matching, and let $S$ be the set of vertices saturated by $f$. Then there exists a set of edges $E' \subseteq \{ e \in E: f(e) > 0\}$ of size $|E'| \le \frac{n}{2}$ such that $S \subseteq \bigcup E'$.
\end{proposition}
\begin{proof} First, we remove from $E$ any edges $e$ with $f(e) = 0$. Next we describe, in three steps, a finite process in which we add new triples to $E$ which are disjoint from  $S$ and assign weights to them, or increase the weights of existing such edges. At the end of the process, we will be able to apply Corollary~\ref{cor:Ex} to the resulting hypergraph.

\vspace{5pt}

\noindent \underline{Step 1}. As long as there are at least $3$ unsaturated vertices in $V$, we pick three of them $u, v, w$, add $uvw$ to $E$ if it is not an edge already, and assign or increase its weight so that at least one of $u, v, w$ becomes saturated while the weight at the others does not exceed $1$.

\vspace{5pt}

\noindent \underline{Step 2}. Let $U$ be the set of vertices that remain unsaturated at the end of Step~1. Then $|U| \le 2$. For each $u \in U$, let $\alpha_u \le 1$ be the additional weight that $u$ needs in order to be saturated. We add to $V$ a set $W = \{w_1, w_2, w_3, w_4\}$ of $4$ new vertices. For each $u \in U$, we add new edges $uw_1w_2$ and $uw_3w_4$ with weight $\frac{\alpha_u}{2}$ each. Now all vertices in $V$ are saturated, and the weight at each $w_i$ is $\sum_{u \in U} \frac{\alpha_u}{2} \le 1$.

\vspace{5pt}

\noindent \underline{Step 3}. We add the $4$ triples contained in $W$ as new edges with weight $\frac{1}{3}(1 - \sum_{u \in U} \frac{\alpha_u}{2})$ each. Now all vertices in $V \cup W$ are saturated.

\vspace{5pt}

Applying Corollary~\ref{cor:Ex} to the resulting hypergraph, we find a set $E''$ of edges (original or new) that cover $V \cup W$, such that $|E''| \le \frac{n+4}{2}$. Because original edges do not contribute to covering $W$, there must be at least two new edges in $E''$. We remove the new edges from $E''$ and obtain $E' \subseteq \{e \in E: f(e) > 0\}$ of size $|E'| \le \frac{n}{2}$. We have $S \subseteq \bigcup E'$ because the new edges are disjoint from $S$.
\end{proof}
We are now ready to state and prove the main result of this section.
\begin{theorem} \label{thm:GSH}
For any $3$-uniform hypergraph $H = (V, E)$ and preference system $\bigl( \succeq_v\bigr)_{v \in V}$ there exists a stable set of edges $\{e_1, e_2, \ldots , e_m\} \subseteq E$ with $m \le \frac{|V|}{2}$.
\end{theorem}
\begin{proof} Let $f: E \to \mathbb{R}_+$ be a stable fractional matching, whose existence is guaranteed by Theorem~\ref{thm:AF}. Let $S$ be the set of vertices saturated by $f$. Applying Proposition~\ref{pro:SAT}, we can find a set of edges $\{e_1, e_2, \ldots , e_m\} \subseteq E$ with $m \le \frac{|V|}{2}$ such that $S \subseteq \bigcup_{i=1}^m e_i$ and $f(e_i) > 0$ for $i=1,2, \ldots ,m$. We claim that $\{e_1, e_2, \ldots , e_m\}$ is stable. Indeed, let $e \in E$. As $f$ is stable, it dominates $e$ at some $v \in e$. By definition, this means that $v \in S$ and all edges $e' \in E_v$ with $f(e') > 0$ satisfy $e' \succeq_v e$. Choosing $e_i$ such that $v \in e_i$, we have $e_i \succeq_v e$, which shows that $\{e_1, e_2, \ldots , e_m\}$ dominates $e$ at $v$. As $e$ was arbitrary, this proves that $\{e_1, e_2, \ldots , e_m\}$ is stable.
\end{proof}
The following example shows that Theorem~\ref{thm:GSH} is sharp.
\begin{example}[$K^{(3)}_4$ with cyclic preferences]
Let $H$ be the complete $3$-uniform hypergraph on $4$ vertices. We write the vertex set as $V = \{0,1,2,3\}$ and use addition modulo $4$. The preferences of each vertex $i$ are: \[ V \setminus \{i+1\} \succeq_i V \setminus \{i+2\} \succeq_i V \setminus \{i+3\}.\] No single edge forms a stable set. Indeed, the edge $V \setminus \{i\}$ does not dominate the edge $V \setminus \{i+3\}$, because the vertices $i+1$ and $i+2$ both prefer $V \setminus \{i+3\}$ to $V \setminus \{i\}$. Thus, the smallest size of a stable set of edges is $2$. More generally, we can take $k$ disjoint copies of this example. Then we have $4k$ vertices, and a stable set of edges must contain at least two edges from each copy, hence a total of at least $2k$ edges.
\end{example}
We remark that the analog of Theorem~\ref{thm:GSH} for graphs holds with the sharp bound $m \le \frac{2}{3}|V|$. This can be deduced quite easily from Tan~\cite{Tan}. We omit the details.

\section*{Acknowledgments}
We are grateful to Ron Aharoni and He Guo for helpful discussions and correspondence.

\end{document}